\pdfoutput=1
\documentclass[12pt, dvipsnames, oneside]{amsart}
\usepackage[margin=1.5in]{geometry}                
\usepackage{hyperref}
\usepackage{graphicx}
\usepackage[table]{xcolor}
\usepackage{float}
\usepackage{multicol}
\usepackage{enumitem}
\usepackage{amssymb}
\usepackage[backend=biber]{biblatex}
\usepackage{epstopdf}
\usepackage{tikz}
\usetikzlibrary{decorations.pathreplacing,calligraphy}
\usepackage{tikz-cd}
\usepackage{linegoal}
\usepackage{changepage}
\usepackage{xfrac}
\usetikzlibrary{decorations.pathmorphing}
\usepackage{rotating}
\usepackage{cleveref}

\newtheorem{theorem}{Theorem}[section]
\newtheorem{theorem/definition}{Theorem/definition}[section]
\newtheorem{corollary}[theorem]{Corollary}
\newtheorem{lemma}[theorem]{Lemma}
\newtheorem{proposition}[theorem]{Proposition}
\theoremstyle{definition}
\newtheorem{observation}[theorem]{Observation}
\newtheorem{remark}[theorem]{Remark}
\newtheorem{algorithm}[theorem]{Algorithm}
\newtheorem{example}{Example}[section]
\newtheorem{definition}[theorem]{Definition}

\title[Quipu quivers and Nakayama algebras]{Quipu quivers and Nakayama algebras with almost separate relations}
\author{Didrik Fosse}

\begin{document}

\begin{abstract}
A Nakayama algebra with almost separate relations is one where the overlap between any pair of relations is at most one arrow. 
In this paper we give a derived equivalence between such Nakayama algebras and path algebras of quivers of a special form known as quipu quivers. Furthermore, we show how this derived equivalence can be used to produce a complete classification of linear Nakayama algebras with almost separate relations.

As an application, we include a list of the derived equivalence classes of all Nakayama algebras of length $\leq 8$ with almost separate relations.

\end{abstract}

\maketitle

\section{Introduction}
Throughout this paper $k$ will be a field and $A_n$ will refer to the linearly ordered quiver of Dynkin type $\mathbb{A}$ with $n$ vertices labelled $1,\ldots, n$.
We will denote by $A_{n,(n_0,\ldots,n_r)}^{(l_0,\ldots, l_r)}$ the linear Nakayama algebra $kA_n/I$ where $I$ is the ideal generated by a relation of length $l_i\geq 2$ beginning in vertex $n_i$ for each $0\leq i\leq r$. We assume that $n_i < n_{i+1}$ and $n_i+l_i< n_{i+1}+l_{i+1}$ for each $0\leq i \leq r-1$, and that $n_r+l_r\leq n$. 

Our focus will be on Nakayama algebras where the overlap between any pair of relations is at most one arrow, meaning $n_{i+1}\geq n_i+l_i-1$ for all $i$, which we will refer to as the algebra having \emph{almost separate relations}.

In order to investigate the derived equivalence classes of these algebras, we will employ the combinatorial quiver mutation developed in \cite{Fosse2022}.

We will show that there exist derived equivalences between linear Nakayama algebras with almost separate relations and so-called quipu algebras. These are path algebras of so-called quipu quivers, whose underlying graphs are what's known as open quipus.
An \emph{open quipu} (or simply quipu) is a tree of maximum degree $3$, such that all vertices of degree $3$ lie on a path. They were introduced by Woo and Neumaier in \cite{WN2007} as an example of graphs with small spectral radius. The name comes from their shape, a line with straight cords going out of it, which is reminiscent of the knotted strings used by the Incas for record keeping purposes. 

The main result of this paper is the existence of derived equivalences between quipu algebras and Nakayama algebras with almost separate relations, as stated in the following theorem. The notation is explained in \cref{section:quipus}. 

\begin{theorem}[\Cref{thm:QuipuToAn}]
Let $P_{(k_0,k_1,\ldots,k_r,k_{r+1})}^{(m_0,m_1,\ldots,m_r)}$ be a quipu, and let 
$$n_i = k_0 + \sum\limits_{j=1}^{i} (m_{j-1} + k_{j} + 1) \text{ for } 1\leq i\leq r+1.$$ 
If $D$ is a quipu quiver, then there exists a derived equivalence
$$kD\simeq A_{n_{r+1},(k_0,n_1,n_2,\ldots, n_{r} )}^{(m_{0}+2, m_{1}+2, \ldots, m_{r}+2)}$$
if and only if $D$ is $P_{(k_0,k_1,\ldots,k_r,k_{r+1})}^{(m_0,m_1,\ldots,m_r)}$ with some orientation.
\end{theorem}

Using this theorem, we are able to give a complete classification up to derived equivalence of linear Nakayama algebras with almost separate relations. Each equivalence class is determined by the shape of a given quipu, with the cords of the quipu corresponding to the relations of the Nakayama algebras in a certian way. We will go through the details in \cref{section:QuipusAndNakayamaAlgebras}.

Relations of length $2$ do not affect the derived equivalence class of a Nakayama algebra with almost separate relations, and the number of almost separate relations of length $\geq 3$ is a derived invariant of such algebras. The length and position of each relation of length $\geq 3$ is also a derived invariant, except for (possibly) the first and last such relations along the quiver.

The precise collection of Nakayama algebras included in each of these derived equivalence classes is stated in the following theorem. Notice that all Nakayama algebras with almost separate relations in a given class can be obtained from each other by a combination of a few simple operations. The operations are:
\begin{itemize}
\item Changing the length and end vertex (resp. start vertex of) the first (resp. last) relation by a certain amount.
\item Adding or removing relations of length $2$.
\item Reversing all arrows in the quiver.
\end{itemize}

This works because all Nakayama algebras in a class correspond to orientations of the same quipu, and the operations correspond to changing that orientation.

\begin{theorem}[\Cref{cor:EquivNakayamaAlgebras}]
If $A_{n,(n_0,\ldots,n_r)}^{(l_0,\ldots, l_r)}$ is a linear Nakayama algebra with almost separate relations, where $l_i\geq 3$ for all $i$, then the following Nakayama algebras are derived equivalent.
\begin{align*}
&A_{n,(n_0,n_1,\ldots,n_{r-1},n_r)}^{(l_0, l_1, \ldots, l_{r-1}, l_r)}  &&A_{n,(n-n_r-l_r+1,n-n_{r-1}-l_{r-1}+1,\ldots,n-n_0-l_0+1)}^{(l_r, l_{r-1}, \ldots, l_{1}, l_0)}\\
&A_{n,(n_0,n_1,\ldots,n_{r-1},n_r)}^{(l_0, l_1, \ldots, l_{r-1}, n-n_r-l_r+3)} &&A_{n,(l_r-2,n-n_{r-1}-l_{r-1}+1,\ldots,n-n_0-l_0+1)}^{(n-n_r-l_r+3, l_{r-1}, \ldots, l_{1}, l_0)}    \\
&A_{n,(l_0-2,n_1,\ldots,n_{r-1},n_r)}^{(n_0+2, l_1, \ldots, l_{r-1}, l_r)}  && A_{n,(n-n_r-l_r+1,n-n_{r-1}-l_{r-1}+1,\ldots,n-n_0-l_0+1)}^{(l_r, l_{r-1}, \ldots, l_{1}, n_0+l_0)}\\
&A_{n,(l_0-2,n_1,\ldots,n_{r-1},n_r)}^{(n_0+2, l_1, \ldots, l_{r-1}, n-n_r-l_r+3)} &&A_{n,(l_r-2,n-n_{r-1}-l_{r-1}+1,\ldots,n-n_0-l_0+1)}^{(n-n_r-l_r+3, l_{r-1}, \ldots, l_{1}, n_0+l_0)}. 
\end{align*}
Also, adding a set of almost separate relations of length $2$ to available positions in any of these algebras will not change the derived equivalence class.
\end{theorem}

\begin{example}
The path algebra of the quipu $P_{(1,2,0,1)}^{(2,1,3)}$ with any orientation is derived equivalent to (among others) the following Nakayama algebras.
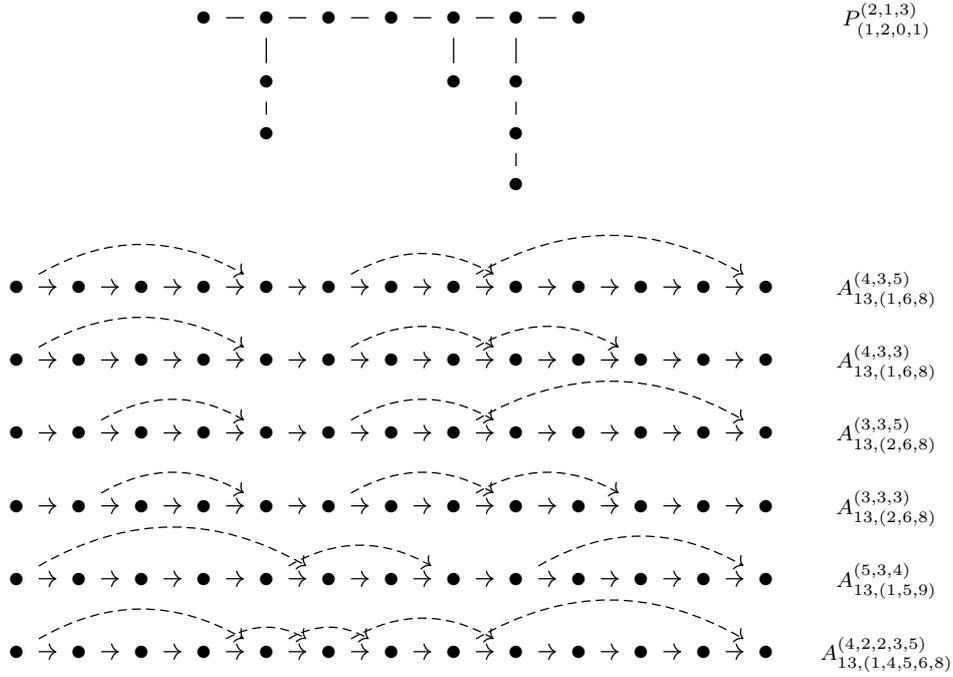
\begin{figure}[H]
\begin{tikzcd}[column sep = tiny, row sep = tiny]
& & & \bullet \ar[r, no head] & \bullet \ar[r, no head]  \ar[d, no head] & \bullet \ar[r, no head] & \bullet \ar[r, no head] & \bullet \ar[r, no head] \ar[d, no head] & \bullet \ar[r, no head] \ar[d, no head] & \bullet & & & & {\scriptstyle P_{(1,2,0,1)}^{(2,1,3)}}\\
& & & & \bullet \ar[d, no head]  & & & \bullet & \bullet \ar[d, no head] & & & & & \\
& & & & \bullet & & & & \bullet \ar[d, no head] & & & & & \\
& & & & & & & & \bullet & & & & & \\
& & & & & & & & & & & & & \\
& & & & & & & & & & & & & \\
& & & & & & & & & & & & & \\
\bullet \arrow[r] \arrow[rrrr, dashed, bend left] & \bullet \arrow[r] & \bullet \arrow[r] & \bullet \arrow[r] & \bullet \arrow[r]  & \bullet \arrow[r] \arrow[rrr, dashed, bend left] & \bullet \arrow[r] & \bullet \arrow[r] \arrow[rrrrr, dashed, bend left] & \bullet \arrow[r] & \bullet \arrow[r] & \bullet \arrow[r] & \bullet \arrow[r] & \bullet & {\scriptstyle A_{13,(1,6,8)}^{(4,3,5)}}\\
\bullet \arrow[r] \arrow[rrrr, dashed, bend left] & \bullet \arrow[r] & \bullet \arrow[r] & \bullet \arrow[r] & \bullet \arrow[r]  & \bullet \arrow[r] \arrow[rrr, dashed, bend left] & \bullet \arrow[r] & \bullet \arrow[r] \arrow[rrr, dashed, bend left] & \bullet \arrow[r] & \bullet \arrow[r] & \bullet \arrow[r] & \bullet \arrow[r] & \bullet & {\scriptstyle A_{13,(1,6,8)}^{(4,3,3)}}\\
\bullet \arrow[r] & \bullet \arrow[r] \arrow[rrr, dashed, bend left] & \bullet \arrow[r] & \bullet \arrow[r] & \bullet \arrow[r]  & \bullet \arrow[r] \arrow[rrr, dashed, bend left] & \bullet \arrow[r] & \bullet \arrow[r] \arrow[rrrrr, dashed, bend left] & \bullet \arrow[r] & \bullet \arrow[r] & \bullet \arrow[r] & \bullet \arrow[r] & \bullet & {\scriptstyle A_{13,(2,6,8)}^{(3,3,5)}}\\
\bullet \arrow[r] & \bullet \arrow[r]  \arrow[rrr, dashed, bend left] & \bullet \arrow[r] & \bullet \arrow[r] & \bullet \arrow[r]  & \bullet \arrow[r] \arrow[rrr, dashed, bend left] & \bullet \arrow[r] & \bullet \arrow[r] \arrow[rrr, dashed, bend left] & \bullet \arrow[r] & \bullet \arrow[r] & \bullet \arrow[r] & \bullet \arrow[r] & \bullet & {\scriptstyle A_{13,(2,6,8)}^{(3,3,3)}}\\
\bullet \arrow[r] \arrow[rrrrr, dashed, bend left] & \bullet \arrow[r] & \bullet \arrow[r] & \bullet \arrow[r] & \bullet \arrow[r] \arrow[rrr, bend left, dashed] & \bullet \arrow[r] & \bullet \arrow[r] & \bullet \arrow[r] & \bullet \arrow[r] \arrow[rrrr, dashed, bend left]& \bullet \arrow[r] & \bullet \arrow[r] & \bullet \arrow[r] & \bullet & {\scriptstyle A_{13,(1,5,9)}^{(5,3,4)}}\\
\bullet \arrow[r] \arrow[rrrr, dashed, bend left] & \bullet \arrow[r] & \bullet \arrow[r] & \bullet \arrow[r] \arrow[rr, bend left, dashed] & \bullet \arrow[r] \arrow[rr, bend left, dashed]  & \bullet \arrow[r] \arrow[rrr, dashed, bend left] & \bullet \arrow[r] & \bullet \arrow[r] \arrow[rrrrr, dashed, bend left] & \bullet \arrow[r] & \bullet \arrow[r] & \bullet \arrow[r] & \bullet \arrow[r] & \bullet & {\scriptstyle A_{13,(1,4,5,6,8)}^{(4,2,2,3,5)}}
\end{tikzcd}
\caption{These Nakayama algebras are all derived equivalent to the path algebra of the quipu $P_{(1,2,0,1)}^{(2,1,3)}$ with any orientation.}
\end{figure}
\end{example}

\section{Quipu quivers}\label{section:quipus}

An (open) quipu is a tree of maximal degree $3$, such that all vertices of degree $3$ lie on a path. 
This path is called the \emph{main string} of the quipu, while the remaining paths going out of each of the vertices of degree $3$ are called \emph{cords}. Both of these terms are taken from \cite{GKS2016}. 

In the notation of \cite{LAN2015}, a quipu can be written as $P_{(k_0,k_1,\ldots, k_r,k_{r+1})}^{(m_0,m_1,\ldots, m_r)}$ where $r\geq 0$ and all $k_i,m_i\geq 0$.
We denote by $m_i$ the length of cord $i$, while $k_i$ denotes the number of vertices along the main string between the beginning of cord $i$ and cord $i+1$.

Note that quipus are not uniquely determined by this notation, and that we do allow trivial cords of length $m_i=0$.

\begin{figure}[H]
\begin{tikzcd}[column sep = small, row sep = small]
\bullet \arrow[r, no head, dotted, "k_0"] & \bullet \arrow[r, no head] & \bullet \arrow[r, no head] \arrow[d, no head] & \bullet \arrow[r, no head, dotted, "k_1"] & \bullet \arrow[r, no head] & \bullet \arrow[r, no head] \arrow[d, no head] & \bullet \arrow[r, no head, dotted] & \bullet \arrow[r, no head] & \bullet \arrow[r, no head] \arrow[d, no head] & \bullet \arrow[r, no head, dotted, "k_r"] & \bullet \arrow[r, no head] & \bullet \arrow[r, no head] \arrow[d, no head] & \bullet \arrow[r, no head, dotted, "k_{r+1}"] & \bullet \\
                                   &                            & \bullet \arrow[d, no head, dotted, "m_0"]            &                                    &                            & \bullet \arrow[d, no head, dotted, "m_1"]            &                                    &                            & \bullet \arrow[d, no head, dotted, "m_{r-1}"]            &                                    &                            & \bullet \arrow[d, no head, dotted, "m_r"]            &                                    &         \\
                                   &                            & \bullet                                       &                                    &                            & \bullet                                       &                                    &                            & \bullet                                       &                                    &                            & \bullet                                       &                                    &        
\end{tikzcd}
\caption{The general form of a quipu $P_{(k_0,k_1,\ldots, k_r,k_{r+1})}^{(m_0,m_1,\ldots, m_r)}$ }
\label{fig:quipu}
\end{figure}
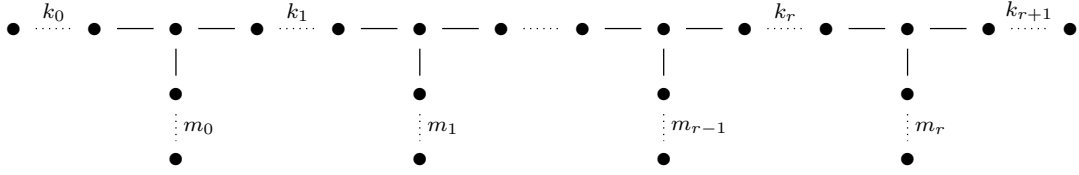

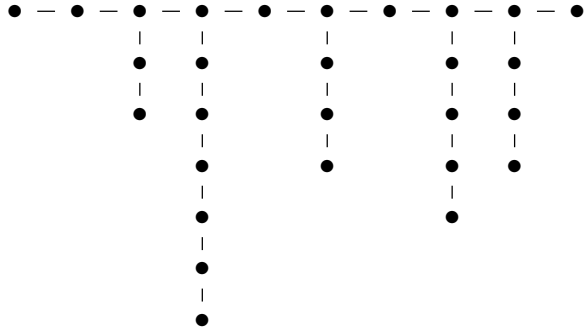
\begin{figure}[H]
\begin{tikzcd}[remember picture, column sep = tiny, row sep = tiny]
\bullet \ar[r, no head] &\bullet \ar[r, no head] & \bullet \ar[r, no head] \ar[d, no head] & \bullet \ar[r, no head] \ar[d, no head] & \bullet \ar[r, no head] & \bullet \ar[r, no head] \ar[d, no head] & \bullet \ar[r, no head]& \bullet \ar[r, no head] \ar[d, no head]& \bullet \ar[r, no head] \ar[d, no head]& \bullet\\
& & \bullet \ar[d, no head] & \bullet \ar[d, no head] & & \bullet \ar[d, no head] & & \bullet \ar[d, no head] & \bullet \ar[d, no head] &\\
& & \bullet & \bullet \ar[d, no head] & & \bullet \ar[d, no head] & & \bullet \ar[d, no head] & \bullet \ar[d, no head] &\\
& & & \bullet \ar[d, no head] & & \bullet & & \bullet \ar[d, no head] & \bullet &\\
& & & \bullet \ar[d, no head] & & & & \bullet  & &\\
& & & \bullet \ar[d, no head] & & & & & &\\
& & & \bullet & & & & & &
\end{tikzcd}\\
\caption{Example of a quipu, this is $P_{(2,0,1,1,0,1)}^{(2,6,3,4,3)}$.}
\label{ex:quipu}
\end{figure}

Because we want to work with quivers, we also need to define an oriented version of a quipu. For a quipu $P_{(k_0,k_1,\ldots, k_r,k_{r+1})}^{(m_0,m_1,\ldots m_r)}$ the \emph{quipu quiver} $D_{(k_0,k_1,\ldots, k_r,k_{r+1})}^{(m_0,m_1,\ldots, m_r)}$ is a directed version of the same graph, where the main string is linearly oriented from left to right and each cord is linearly oriented away from the main string. We will call the path algebra of a quipu quiver a quipu algebra.

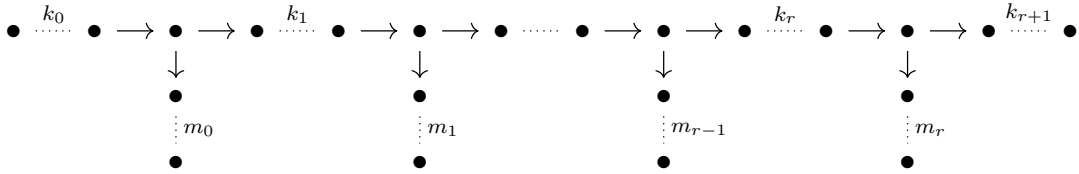
\begin{figure}[H]
\begin{tikzcd}[column sep = small, row sep = small]
\bullet \arrow[r, no head, dotted, "k_0"] & \bullet \arrow[r] & \bullet \arrow[r] \arrow[d] & \bullet \arrow[r, no head, dotted, "k_1"] & \bullet \arrow[r] & \bullet \arrow[r] \arrow[d] & \bullet \arrow[r, no head, dotted] & \bullet \arrow[r] & \bullet \arrow[r] \arrow[d] & \bullet \arrow[r, no head, dotted, "k_r"] & \bullet \arrow[r] & \bullet \arrow[r] \arrow[d] & \bullet \arrow[r, no head, dotted, "k_{r+1}"] & \bullet \\
                                   &                            & \bullet \arrow[d, no head, dotted, "m_0"]            &                                    &                            & \bullet \arrow[d, no head, dotted, "m_1"]            &                                    &                            & \bullet \arrow[d, no head, dotted, "m_{r-1}"]            &                                    &                            & \bullet \arrow[d, no head, dotted, "m_r"]            &                                    &         \\
                                   &                            & \bullet                                       &                                    &                            & \bullet                                       &                                    &                            & \bullet                                       &                                    &                            & \bullet                                       &                                    &        
\end{tikzcd}
\caption{The general form of a quipu quiver $D_{(k_0,k_1,\ldots, k_r,k_{r+1})}^{(m_0,m_1,\ldots, m_r)}$ }
\end{figure}

By looking at \cref{fig:quipu} we can immediately make the following two observations: 

\begin{observation}
The number of vertices in the quipu $P_{(k_0,k_1,\ldots, k_r,k_{r+1})}^{(m_0,m_1,\ldots, m_r)}$ is equal to $r+1+\sum\limits_{i=0}^{r+1}k_i +  \sum\limits_{i=0}^r m_i$.
\end{observation}

\begin{observation}\label{obs:QuipuOrientations}
There exist graph isomorphisms
\begin{align*}
P_{(k_0,k_1,\ldots, k_{r-1},k_r,k_{r+1})}^{(m_0,m_1,\ldots ,m_{r-1},m_r)} &\simeq P_{(m_0,k_1,\ldots, k_{r-1}, k_r, k_{r+1})}^{(k_0,m_1,\ldots, m_{r-1}, m_r)}\\
&\simeq P_{(k_0,k_1,\ldots, k_{r-1}, k_r,m_{r})}^{(m_0,m_1,\ldots, m_{r-1}, k_{r+1})}\\
&\simeq P_{(k_{r+1},k_r,k_{r-1}\ldots, k_1,k_{0})}^{(m_r,m_{r-1},\ldots, m_1, m_0)}
\end{align*}
\end{observation}

The second observation essentially says that a quipu is not uniquely defined in this notation. We can view the first cord as a part of the main string, making the first part of the main string a cord instead, and the same is true for the last cord and last part of the main string. Also reversing the order of the main string and the cords clearly doesn't change quipu. 

By combining these operations we get a total of $8$ expressions for the same quipu, some of which may coincide depending on the shape of the ends of the quipu. Note that even if different expressions define the same quipu, their corresponding quipu quivers will in general not be isomorphic.

\begin{figure}[H]\label{ex:nonuniqueQuipu}
\begin{tikzcd}[remember picture, column sep = small, row sep = small]
\bullet \ar[r] & \bullet \ar[r] & \bullet \ar[r] \ar[d] & \bullet &  \bullet \ar[r] & \bullet \ar[r] \ar[d] & \bullet & \bullet \ar[r] & \bullet \ar[r] \ar[d] & \bullet \ar[r] & \bullet \\
& & \bullet & & & \bullet \ar[d] & & & \bullet & & \\
& & & & & \bullet & & & & & 
\end{tikzcd}\\
\caption{Three different quipu quivers which are orientations of the same quipu: $D_{(2,1)}^{(1)}$, $D_{(1,1)}^{(2)}$ and $D_{(1,2)}^{(1)}$.}
\end{figure}
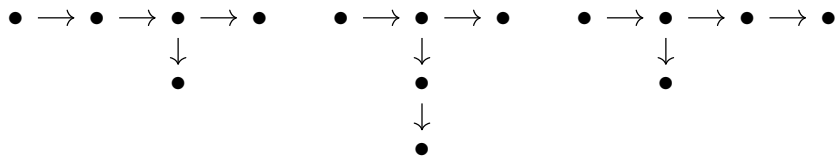

\section{Cord/relation-swap}\label{section:CR-swap}
In this section we introduce the notion of a quipu quiver with almost separate relations, and we develop a technique for going between such quivers in a way that preserves the derived equivalence class of their path algebras. Later we will use this to find derived equivalences between quipu algebras with no relations and Nakayama algebras with almost separate relations.

\begin{definition}[Quipu quiver with almost separate relations]
We say that a quipu quiver has \emph{almost separate relations} if it comes with a set of zero relations satisfying the following conditions. 
\begin{itemize}
\item The relations must all be along the main string
\item No cord may begin at an internal vertex of a relation
\item No pair of relations may have an overlap of more than one arrow
\end{itemize}
\end{definition}

We will now show a technique for generating derived equivalent quipus with almost separate relations. We do this by repeatedly applying the quiver mutation defined in \cite{Fosse2022}, which modifies quivers with relations in a way that preserves the derived equivalence class of their path algebras. We will not show the mutation procedure in detail here, see the paper for more information.

\begin{algorithm}[Cord/relation-swap]\label{algorithm:CRswap}
Let $Q$ be a quipu quiver with almost separate relations, and let $r$ be a relation in $Q$ with length $l\geq 2$, start vertex $s(r)$ and end vertex $t(r)$. Let $m\geq 0$ denote the length of the (possibly trivial) cord starting in $s(r)$. We construct a new quipu with almost separate relations $Q'$ from $Q$ as follows:
\begin{enumerate}[label=\arabic*.]
\item Remove the relation $r$ and the cord of length $m$ starting in $s(r)$.
\item Replace the path $s(r)\rightarrow \ldots \rightarrow t(r)$ of length $l$ by a path of length $m+2$. If there is a relation starting in the vertex before $t(r)$, keep it.
\item If there is an arrow ending in $s(r)$, add a relation of length $m+2$ from the start vertex of this arrow to the vertex one step before $t(r)$ along the main string.
\item Add a cord of length $l-2$ going out of the vertex before $t(r)$.
\end{enumerate}
\end{algorithm}

\begin{proposition}
The quipu quiver $Q'$ with relations constructed in \cref{algorithm:CRswap} has a path algebra which is derived equivalent to that of the quipu quiver Q with relations.
\end{proposition}

This move can be viewed as the cord and relation swapping order and swapping lengths ($\pm\ 2$), hence we refer to it as \emph{cord/relation-swap} (CR-swap for short).

Below is a visual representation of the CR-swap in action.\\

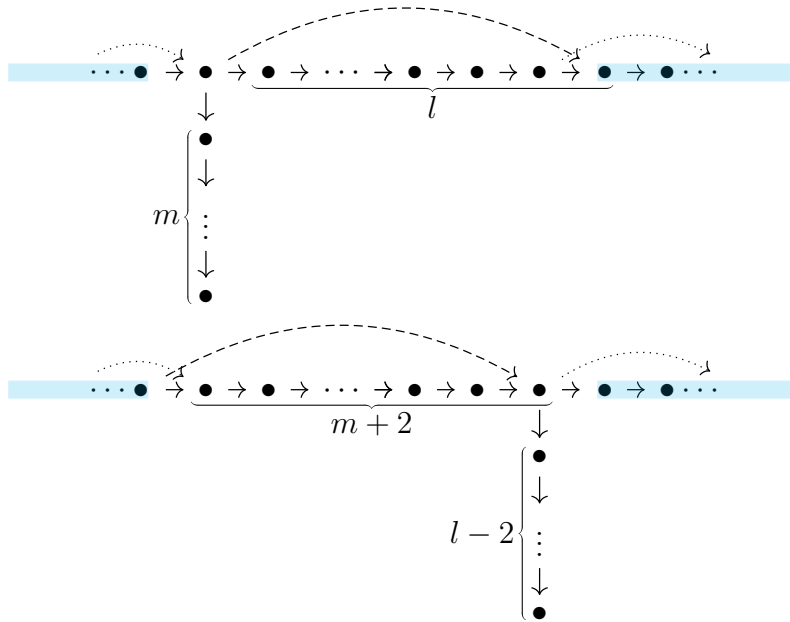
\begin{figure}[H]
\begin{tikzcd}[remember picture, column sep = tiny, row sep = small]
\tikz[baseline,remember picture]\node[anchor=base,inner sep=0pt] (node1) {}; \hspace{1cm} \cdots \ar[rr, bend left, dotted] & [-15] \bullet \tikz[baseline,remember picture]\node[anchor=base,inner sep=0pt] (node2) {}; \ar[r] & \bullet \ar[r] \ar[d] \ar[rrrrrr, dashed, bend left] & \bullet \ar[r] & \cdots \ar[r]  \ar[r] & \bullet \ar[r] &\bullet \ar[r] & \bullet \ar[r] \ar[rrr, bend left, dotted] & \tikz[baseline,remember picture]\node[anchor=base,inner sep=0pt] (node3) {}; \bullet \ar[r] & \bullet & [-15] \cdots \hspace{1cm} \tikz[baseline,remember picture]\node[anchor=base,inner sep=0pt] (node4) {};\\
& & \bullet \ar[d] & & & & & & &\\
& & \vdots \ar[d] & & & & & & &\\
& & \bullet & & & & & & & \\
& & & & & & & & & \\
\tikz[baseline,remember picture]\node[anchor=base,inner sep=0pt] (node5) {}; \hspace{1cm} \cdots \ar[rr, bend left, dotted] & [-15] \bullet \tikz[baseline,remember picture]\node[anchor=base,inner sep=0pt] (node6) {};  \ar[r] \ar[rrrrrr, dashed, bend left] & \bullet \ar[r] & \bullet \ar[r] & \cdots \ar[r]  \ar[r] & \bullet \ar[r] &\bullet \ar[r] & \bullet \ar[r] \ar[d] \ar[rrr, dotted, bend left] & \tikz[baseline,remember picture]\node[anchor=base,inner sep=0pt] (node7) {};  \bullet \ar[r] & \bullet & [-15] \cdots \hspace{1cm} \tikz[baseline,remember picture]\node[anchor=base,inner sep=0pt] (node8) {}; \\
& & & & & & & \bullet \ar[d] & &\\
& & & & & & & \vdots \ar[d] & &\\
& & & & & & & \bullet & &
\end{tikzcd}

\begin{tikzpicture}[overlay, remember picture]
   \draw [decorate, decoration = {calligraphic brace}] (-2.8,4.4) --  (-2.8,6.7) node[midway,left] {$m$};
   \draw [decorate, decoration = {calligraphic brace}] (2.75,7.3) --  (-2,7.3) node[midway,below] {$l$};
   \draw[cyan,-,double=cyan, double distance = 16\pgflinewidth, opacity=0.1] (node1) -- (node2);
   \draw[cyan,-,double=cyan, double distance = 16\pgflinewidth, opacity=0.1] (node3) -- (node4);
   \draw [decorate, decoration = {calligraphic brace}] (1.6,0.21) --  (1.6,2.5) node[midway,left] {$l-2$};
   \draw [decorate, decoration = {calligraphic brace}] (1.95,3.1) --  (-2.8,3.1) node[midway,below] {$m+2$};
   \draw[cyan,-,double=cyan, double distance = 16\pgflinewidth, opacity=0.1] (node5) -- (node6);
   \draw[cyan,-,double=cyan, double distance = 16\pgflinewidth, opacity=0.1] (node7) -- (node8);
\end{tikzpicture}\\
\caption{Relevant part of the quipu quiver before and after the CR-swap. The dotted arrows are optional relations and the shaded parts of the quiver remain unchanged.}
\end{figure}

\begin{proof}[Proof of proposition]
Starting from a quipu quiver with almost separate relations, we will mutate through a sequence of other quivers, until we arrive at another quipu quiver with almost separate relations.
We will explicitly state which vertices to mutate at in order to produce the desired change in the quiver. We begin by numbering the vertices like this
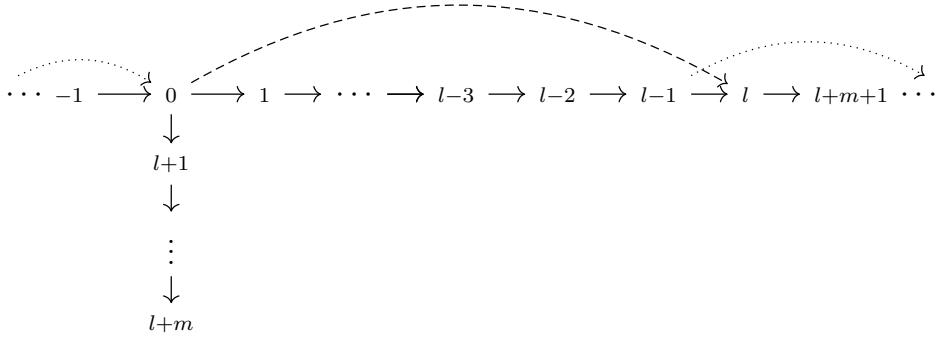
\begin{equation*}
\hspace{-1cm}
\begin{tikzcd}[remember picture, column sep = small, row sep = small]
\hspace{1cm} \cdots \ar[rr, bend left, dotted] & [-20] {\scriptstyle -1} \ar[r] & {\scriptstyle 0} \ar[r] \ar[d] \ar[rrrrrr, dashed, bend left] & {\scriptstyle 1} \ar[r] & \cdots \ar[r]  \ar[r] & {\scriptstyle l-3} \ar[r] & {\scriptstyle l-2} \ar[r] & {\scriptstyle l-1} \ar[r] \ar[rrr, bend left, dotted] & {\scriptstyle l} \ar[r] & {\scriptstyle l+m+1} & [-20] \cdots \hspace{1cm} \\
& & {\scriptstyle l+1} \ar[d] & & & & & & &\\
& & \vdots \ar[d] & & & & & & &\\
& & {\scriptstyle l+m} & & & & & & &
\end{tikzcd}
\end{equation*}
The dotted arrows represent optional zero relations to/from some vertices in other parts of the quiver. We include these relations to show that they remain unchanged after the CR-swap. If there are no such relations, the CR-swap still works the same, which is easily seen by repeating the following proof without the dotted relations.

With this numbering, we will mutate at the following sequence of vertices:
\begin{equation*}
{\scriptstyle l-2},\ {\scriptstyle l-3},\ \ldots,\ {\scriptstyle 1},\ {\scriptstyle 0},\ {\scriptstyle l+1},\ {\scriptstyle l+2},\ \ldots,\ {\scriptstyle l+m}
\end{equation*}
In other words, we begin by mutating on the second vertex from the end of the relation. Then we move one vertex back each time, until we hit the start of the relation, and then we mutate along the cord. 

Now, let us mutate the quiver at vertex ${\scriptstyle l-2}$. We do this by applying the mutation rules, which in this case just means adding an arrow from ${\scriptstyle l-3}\rightarrow {\scriptstyle l-1}$, flipping the arrow ${\scriptstyle l-2}\rightarrow {\scriptstyle l-1}$, and replacing the arrow ${\scriptstyle l-3}\rightarrow {\scriptstyle l-2}$ by a relation. Thus, we end up with the quiver
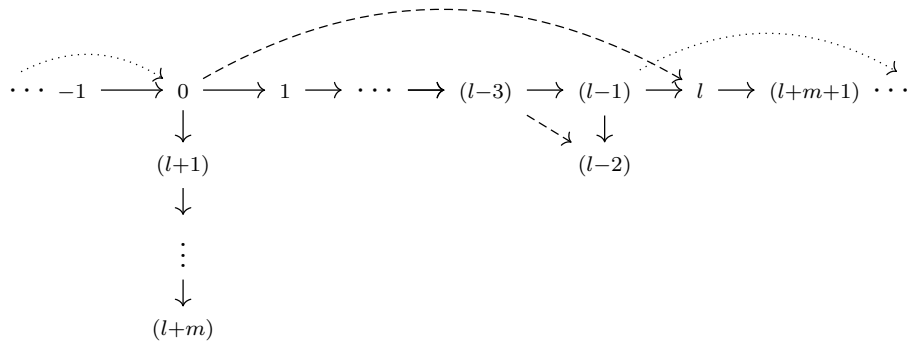
\begin{equation*}
\begin{tikzcd}[remember picture, column sep = small, row sep = small]
\hspace{1cm} \cdots \ar[rr, bend left, dotted] & [-20] {\scriptstyle -1} \ar[r] & {\scriptstyle 0} \ar[r] \ar[d] \ar[rrrrr, dashed, bend left] & {\scriptstyle 1} \ar[r] & \cdots \ar[r]  \ar[r] & {\scriptstyle (l-3)} \ar[r] \ar[dr, dashed] & {\scriptstyle (l-1)} \ar[r] \ar[d] \ar[rrr, bend left, dotted] & {\scriptstyle l} \ar[r] & {\scriptstyle (l+m+1)} & [-20] \cdots \hspace{1cm} \\
& & {\scriptstyle (l+1)} \ar[d] & & & & {\scriptstyle (l-2)} & & &\\
& & \vdots \ar[d] & & & & & & &\\
& & {\scriptstyle (l+m)} & & & & & & &
\end{tikzcd}
\end{equation*}
Mutation at vertex ${\scriptstyle l-3}$ is now done similarly, except we also turn the relation ${\scriptstyle l-3}\dashrightarrow {\scriptstyle l-2}$ into an arrow, and add a relation which makes the new composition ${\scriptstyle l-1}\rightarrow {\scriptstyle l-3}\rightarrow {\scriptstyle l-2}$ equal to the arrow ${\scriptstyle l-1}\rightarrow {\scriptstyle l-2}$. All together, these changes result in the quiver
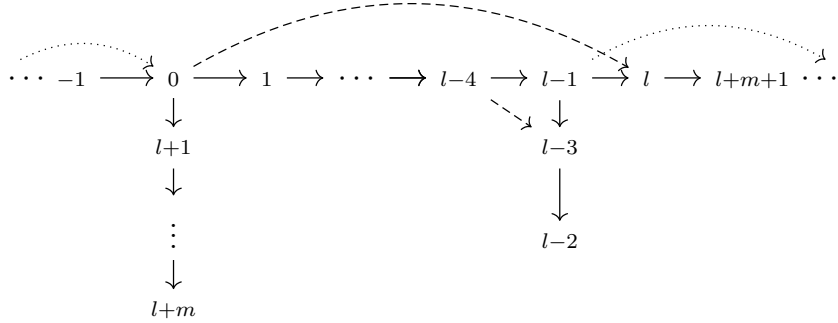
\begin{equation*}
\begin{tikzcd}[remember picture, column sep = small, row sep = small]
\hspace{1cm} \cdots \ar[rr, bend left, dotted] & [-20] {\scriptstyle -1} \ar[r] & {\scriptstyle 0} \ar[r] \ar[d] \ar[rrrrr, dashed, bend left] & {\scriptstyle 1} \ar[r] & \cdots \ar[r]  \ar[r] & {\scriptstyle l-4} \ar[r] \ar[dr, dashed] & {\scriptstyle l-1} \ar[r] \ar[d] \ar[rrr, bend left, dotted] & {\scriptstyle l} \ar[r] & {\scriptstyle l+m+1} & [-20] \cdots \hspace{1cm} \\
& & {\scriptstyle l+1} \ar[d] & & & & {\scriptstyle l-3} \ar[d] & & &\\
& & \vdots \ar[d] & & & & {\scriptstyle l-2} & & &\\
& & {\scriptstyle l+m} & & & & & & &
\end{tikzcd}
\end{equation*}
This pattern continues as we mutate along ${\scriptstyle l-4},\ \ldots, {\scriptstyle 1}$. With each mutation the relation on top becomes one arrow shorter, and the path going out of vertex ${\scriptstyle l-1}$ becomes one arrow longer. This is easily seen using the mutation rules. Eventually, after mutating at each vertex from ${\scriptstyle l-2}$ down to and including ${\scriptstyle 1}$, we have this quiver
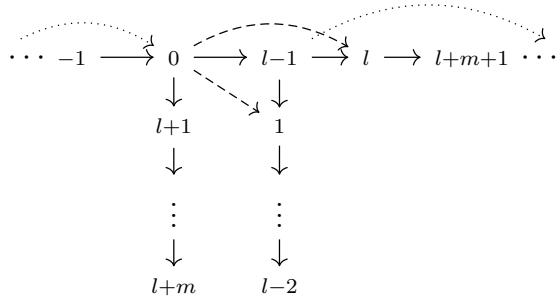
\begin{equation*}
\begin{tikzcd}[remember picture, column sep = small, row sep = small]
\hspace{1cm} \cdots \ar[rr, bend left, dotted] & [-20] {\scriptstyle -1} \ar[r] & {\scriptstyle 0} \ar[r] \ar[d] \ar[rr, dashed, bend left] \ar[dr, dashed] & {\scriptstyle l-1} \ar[r] \ar[d] \ar[rrr, bend left, dotted] & {\scriptstyle l} \ar[r] & {\scriptstyle l+m+1} & [-20] \cdots \hspace{1cm} \\
& & {\scriptstyle l+1} \ar[d] & {\scriptstyle 1} \ar[d] & & &\\
& & \vdots \ar[d] & \vdots \ar[d] & & &\\
& & {\scriptstyle l+m} & {\scriptstyle l-2} & & &
\end{tikzcd}
\end{equation*}
Mutating at vertex ${\scriptstyle 0}$ will now yield, after some rearranging and simplifying, the following quiver. Note that the dotted relation on the left side gives rise to two new relations, both starting at the same vertex.

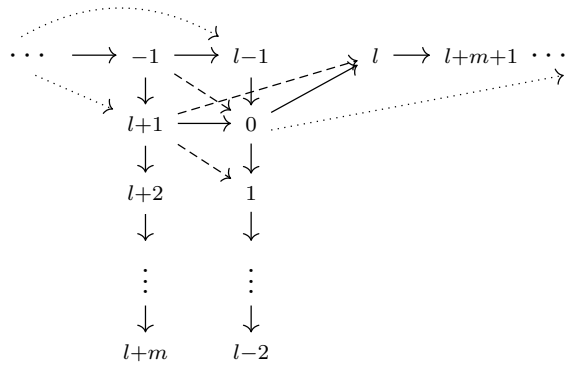
\begin{equation*}
\begin{tikzcd}[remember picture, column sep = small, row sep = small]
\hspace{1cm} \cdots \ar[rrr, bend left, dotted] \ar[drr, dotted] & [-20] {} \ar[r] & {\scriptstyle -1} \ar[r] \ar[d] \ar[dr, dashed] & {\scriptstyle l-1} \ar[d] & & {\scriptstyle l} \ar[r] & {\scriptstyle l+m+1} & [-20] \cdots \hspace{1cm} \\
& & {\scriptstyle l+1} \ar[r] \ar[d] \ar[urrr, dashed] \ar[dr, dashed] & {\scriptstyle 0} \ar[d] \ar[urr] \ar[urrrr, dotted, bend right, looseness=0]& & & &\\
& & {\scriptstyle l+2} \ar[d] & {\scriptstyle 1} \ar[d] & & & &\\
& & \vdots \ar[d] & \vdots \ar[d] & & & &\\
& & {\scriptstyle l+m} & {\scriptstyle l-2} & & & &
\end{tikzcd}
\end{equation*}
Mutating at ${\scriptstyle l+1}$ now yields
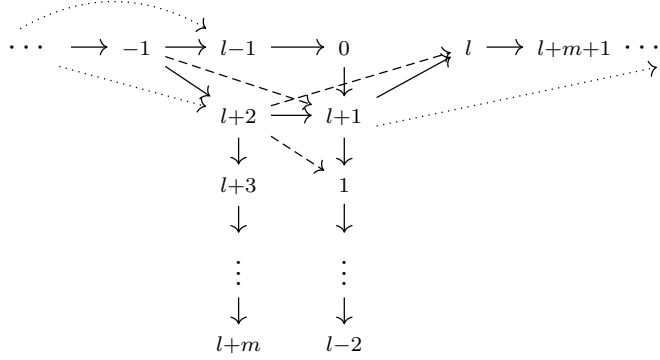
\begin{equation*}
\begin{tikzcd}[remember picture, column sep = small, row sep = small]
\hspace{1cm} \cdots \ar[rrr, bend left, dotted] \ar[drrr, dotted] & [-20] {} \ar[r] & {\scriptstyle -1} \ar[r] \ar[dr] \ar[drr, dashed] & {\scriptstyle l-1} \ar[r] & {\scriptstyle 0} \ar[d] & & {\scriptstyle l} \ar[r] & {\scriptstyle l+m+1} & [-20] \cdots \hspace{1cm} \\
& & & {\scriptstyle l+2} \ar[r] \ar[d] \ar[urrr, dashed] \ar[dr, dashed] & {\scriptstyle l+1} \ar[d] \ar[urr] \ar[urrrr, dotted, bend right, looseness=0]& & & &\\
& & & {\scriptstyle l+3} \ar[d] & {\scriptstyle 1} \ar[d] & & & &\\
& & & \vdots \ar[d] & \vdots \ar[d] & & & &\\
& & & {\scriptstyle l+m} & {\scriptstyle l-2} & & & &
\end{tikzcd}
\end{equation*}

Notice that the structure of the quiver is mostly unchanged, except ${\scriptstyle l+1}$ has taken the place of ${\scriptstyle 0}$, ${\scriptstyle l+2}$ has taken the place of ${\scriptstyle l+1}$, and ${\scriptstyle 0}$ has been inserted into the upper path of the commutative diagram. So the difference is that the cord going out of the lower path of the commutativity relation has become shorter by one arrow, and the upper path has become longer by one arrow.

This pattern continues as we keep mutating along the cord ${\scriptstyle l+2},\ {\scriptstyle l+3},\ \ldots$, and eventually, after mutating at ${\scriptstyle l+m-1}$, we have the quiver
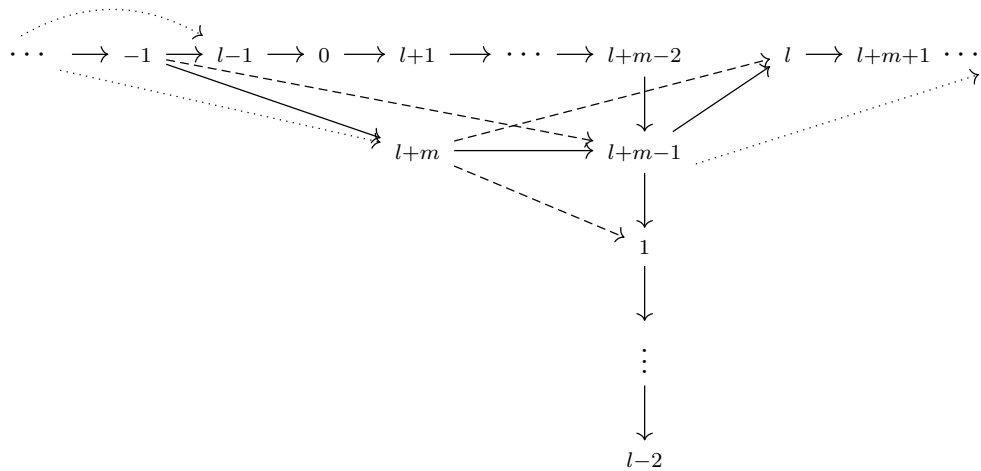
\begin{equation*}
\begin{tikzcd}[remember picture, column sep = small]
\hspace{1cm} \cdots \ar[rrr, bend left, dotted] \ar[drrrrr, dotted] & [-20] {} \ar[r] & {\scriptstyle -1} \ar[r] \ar[drrr] \ar[drrrrr, dashed] & {\scriptstyle l-1} \ar[r] & {\scriptstyle 0} \ar[r] & {\scriptstyle l+1} \ar[r] & \cdots \ar[r] & {\scriptstyle l+m-2} \ar[d] & & {\scriptstyle l} \ar[r] & {\scriptstyle l+m+1} & [-20] \cdots \hspace{1cm} \\
& & & & & {\scriptstyle l+m} \ar[rr]  \ar[urrrr, dashed] \ar[drr, dashed] & & {\scriptstyle l+m-1} \ar[d] \ar[urr] \ar[urrrr, dotted, bend right, looseness=0] & & &  &\\
& & {}  & & & & & {\scriptstyle 1} \ar[d] & & & &\\
& &   & & & & & \vdots \ar[d] &  & & &\\
& & {} & & & & & {\scriptstyle l-2}  & & & &
\end{tikzcd}
\end{equation*}

Finally, if we now mutate at ${\scriptstyle l+m}$, and do a bit of rearranging, we end up with the following quiver
\begin{equation*}
\begin{tikzcd}[remember picture, column sep = tiny, row sep = small]
\hspace{1cm} \cdots \ar[rr, dotted, bend left] & [-15] {\scriptstyle -1} \ar[r] \ar[rrrrrr, bend left, dashed, looseness=0.75] & {\scriptstyle l-1} \ar[r] & {\scriptstyle 0} \ar[r] & {\scriptstyle l+1} \ar[r] & \cdots \ar[r] & {\scriptstyle l+m-1} \ar[r] & {\scriptstyle l+m} \ar[r] \ar[d] \ar[rrr, bend left, dotted] & {\scriptstyle l} \ar[r] & {\scriptstyle l+m+1} & [-15] \cdots \hspace{1cm} \\
& & & & & & & {\scriptstyle 1} \ar[d] & & & \\
& & & & & & & {\scriptstyle 2} \ar[d] & & & \\
& & & & & & & \vdots \ar[d] & & & \\
& & & & & & & {\scriptstyle l-2} & & &
\end{tikzcd}
\end{equation*}
This is what we wanted to show. The cord and relation have swapped places, and it is easy to see that the length of the cord is $l-2$, and the length of the relation is $m+2$. Notice that relatively to the rest of the quiver, the dotted relations remain unchanged. Also note that if there is no arrow ${\scriptstyle -1}\rightarrow {\scriptstyle 0}$, then the mutation simplifies, and we get no relation after the swap. This completes the proof.
\end{proof}

The CR-swap is a valuable tool for investigating derived equivalence classes of quipu algebras. We now give a corollary for some special cases of using the CR-swap.

\begin{corollary}\label{cor:CRnocord}
Let $Q$ be a quipu quiver with almost separate relations, with a relation of length $l$ and a cord of length $m$ starting at the same vertex.
If we apply the CR-swap to this relation and cord, then the following holds:
\begin{enumerate}
	\item If $l=2$, then the resulting cord obtained by the CR-swap has length $0$. 
	\item If $m=0$, then the resulting relation obtained by the CR-swap has length $2$ 
\end{enumerate}

\end{corollary}

Point $1$ means the resulting quipu quiver after the CR-swap in that case just has a relation of length $m+2$, shifted one vertex back along the main string, with no cord going out of its last vertex. From point $2$, we see that if we swap at a relation which has no cord going out of its first vertex, then the resulting quipu quiver has a relation of length $2$, with a cord of length $l-2$ going out of its last vertex.

By combining these two points, we observe that CR-swapping a lone relation of length $2$ with no cord, will result in a new relation of length $2$ with no cord, shifted one vertex back in the quiver.

\begin{remark}
As presented here, the CR-swap uses right tilting mutation. There is an entirely dual result using left mutation, meaning we can swap cords and relations in both directions.
\end{remark}

\section{Quipus and Nakayama algebras}\label{section:QuipusAndNakayamaAlgebras}

In this section we will show that any linear Nakayama algebra with almost separate relations is derived equivalent to the path algebra of some quipu quiver. Furthermore we use this to develop a method to classify all Nakayama algebras with almost separate relations into derived equivalence classes.
 
The main idea is to find a way to mutate between a quiver with a relation of a certain length at a certian position, and one with a cord of a corresponding length at a corresponding position. By combining this with the CR-swap, we will be able to construct a sequence of quiver mutations between a quipu quiver and the quiver of a Nakayama algebra.

We will show that there exists a sequence of quiver mutations between the quipu quiver $D_{(k_0,k_1,\ldots,k_r,k_{r+1})}^{(m_0,m_1,\ldots,m_r)}$ and the quiver $A_{n}$ with $(r+1)$ relations, where the $i$-th relation has length $(m_i+2)$ and its starting vertex is determined by counting the number of vertices appearing before the beginning of the $i$-th cord in $D_{(k_0,k_1,\ldots,k_r,k_{r+1})}^{(m_0,m_1,\ldots,m_r)}$.
So if the vertices in $A_n$ are labelled from $1$ to $n$, the $0$-th relation will begin at vertex $k_0$ and have length $m_0$, the $1$-st relation will begin at vertex $(k_0+m_0+k_1+1)$ with length $m_1$, and so on. In general relation $i$ will have length $m_i$ and begin at vertex $k_0 + \sum\limits_{j=1}^{i} (m_{j-1} + k_{j} + 1)$ for $1\leq i\leq r$.
See \cref{fig:quipuLineCorrespondence} for an example of this correspondence.

\begin{figure}[H]
\begin{tikzcd}[column sep = tiny, row sep = tiny]
\bullet \arrow[r] & \bullet \arrow[r] \arrow[rrrr, dashed, bend left] & \bullet \arrow[r] & \bullet \arrow[r] & \bullet \arrow[r]  & \bullet \arrow[r] & \bullet \arrow[r] \arrow[rrrrr, dashed, bend left] & \bullet \arrow[r] & \bullet \arrow[r] & \bullet \arrow[r] & \bullet \arrow[r] \arrow[rrr, dashed, bend left] & \bullet \arrow[r] & \bullet \arrow[r] & \bullet\\
& & & & & & \rotatebox[origin=c]{270}{$\leftrightsquigarrow$} & & & & & \\
& & \bullet \arrow[r] & \bullet \arrow[r] & \bullet \arrow[r] \arrow[d] & \bullet \arrow[r] & \bullet \arrow[r] & \bullet \arrow[r] \arrow[d]& \bullet \arrow[r] \arrow[d] & \bullet & & & &\\
& & & & \bullet \arrow[d] & &  & \bullet \arrow[d] & \bullet &  &  & & &\\
& & & & \bullet & &  &  \bullet \arrow[d] & &  &  & & &\\
& & & & & &  & \bullet & & & & & &      
\end{tikzcd}
\caption{There is a derived equivalence between the algebras $A_{14,(2,7,11)}^{(4,5,3)}$ and $kD_{(2,2,0,1)}^{(2,3,1)}$.}
\label{fig:quipuLineCorrespondence}
\end{figure}
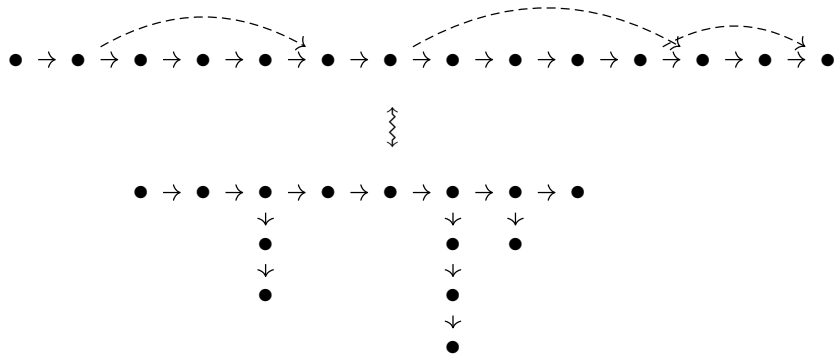

\begin{lemma}\label{lemma:firstRelToCord}
Let $Q$ be a quipu quiver with almost separate relations, and let $s,\ t$ and $l$ denote resp. the source, target and length of the first relation along the main string of $Q$. Let $Q'$ be identical to $Q$, except the following changes:
\begin{itemize}
\item The relation $s\dashrightarrow t$ is removed.
\item The path $s\rightarrow \ldots \rightarrow t$ is replaced by a path $s\rightarrow t'\rightarrow t$.
\item There is a cord of length $l-2$ going out of vertex $t'$.
\end{itemize}
Then $Q$ and $Q'$ have derived equivalent path algebras.
\end{lemma}

\begin{proof}
The idea behind this proof is to repeatedly apply the CR-swap to the first relation along the main string in $Q$. Each time we do so, what constitutes the first relation is moved one step towards the beginning of the main string, until eventually it is removed entirely. Along the way the cords to the left of the first relation will be turned into relations and then back into cords, leaving everything except the first relation unchanged in the end. What was the first relation has then been replaced by a single arrow with a cord going out of its target vertex.

Let $m_i\geq 0$ be the length of the (possibly trivial) cord starting in vertex number $i$ of the main string for each $1\leq i\leq n_0$. If the first relation of $Q$ begins in vertex $n_0$ and has length $l_0$, we can apply the CR-swap to it with $l=l_0$ and $m=m_{n_0}$.
This will yield a quiver which is the same as before, except the given cord and relation have been removed, and replaced by a relation of length $m_{n_0}+2$ starting in vertex $n_0-1$ with a cord of length $l_0-2$ going out of its last vertex.

Since vertex $n_0-1$ by assumption also has a cord of length $m_{n_0-1}$ going out of it, we can now apply the CR-swap to this new relation with $l=m_{n_0}+2$ and $m=m_{n_0-1}$. This will yield a relation of length $m_{n_0-1}+2$ starting in vertex $n_0-2$, with a cord of length $(m_{n_0}+2)-2=m_{n_0}$ going out of its last vertex, which is $n_0$. In other words, we have recreated the cord of length $m_{n_0}$ going out of vertex $n_0$, but the first relation on the main string now appears before this cord, rather than after. And from vertex $n_0$ there is now a single arrow to the start vertex of the cord of length $l_0-2$ we constructed in the first step.

Because each vertex from $1$ to $n_0$ has a cord of length $\geq 0$ we can repeat this process. Each time we apply the CR-swap to the current first relation on the main string, it is removed and replaced by a new relation starting one vertex further back. Also, the cord we removed in the previous step is reformed just after this new relation. After $n_{1}$ CR-swaps the relation is removed from the first relation of the main string, and all the cords have been restored. This leaves us with a quiver isomorphic to the one we started with, except the relation of length $l_0$ starting in vertex $n_0$ has been replaced by a cord of length $l_0-2$ starting in vertex $n_0+1$
\end{proof}

\begin{theorem}\label{thm:AnToQuipu}
Let $A_{n,(n_0,n_1,\ldots, n_{r} )}^{(l_{0}, l_{1}, \ldots, l_{r})}$ be a linear Nakayama algebra with almost separate relations. Then there exists a derived equivalence between $A_{n,(n_0,n_1,\ldots, n_{r} )}^{(l_{0}, l_{1}, \ldots, l_{r})}$ and the path algebra of the quipu quiver 
$$D_{(\hspace{0.35em} n_0 \hspace{0.35em},\ (n_1+1-n_0-l_0),\ \ldots\ ,\ (n_r + 1 - n_{r-1}-l_{r-1}),\ (n + 1 - n_r-l_r))}^{(l_0-2,\hspace{1.75em}  l_1-2 \hspace{1.75em},\ \ldots\ , \hspace{2.6em}  l_r-2 \hspace{2.6em})}$$
\end{theorem}

\begin{proof}
Since $A_{n,(n_0,n_1,\ldots, n_{r} )}^{(l_{0}, l_{1}, \ldots, l_{r})}$ is a has almost separate relations, its quiver can be considered a quipu quiver with almost separate relations. This means that the assumptions in \cref{lemma:firstRelToCord} are satisfied, and applying the CR-swap $n_0$ times to the relation starting in vertex $n_0$ will replace that relation by a cord of length $l_0-2$ going out of vertex $n_0+1$. 

Because the internal vertices of the removed relation have been moved out to form the cord, all subsequent vertices in the main string have had their distance from the beginning of the main string reduced by $l_0-2$. As a consequence the next relation along the main string, whose start vertex was $n_1$ before the swaps, now starts in vertex $n_1-l_0+2$ along the main string. Again, applying the CR-swap $n_1-l_0+2$ times to this relation will remove it and replace it by a cord of length $l_1-2$ starting in vertex $n_1-l_0+3$. Notice that the number of vertices strictly between the starting vertices for these two cords is 
$$(n_1-l_0+3)-(n_0+1)-1=n_1+1-n_0-l_0,$$
which is equal to one plus the number of vertices between the end of the first and the beginning of the second relation.

We can repeat this process for the remaining relations in the quiver, each time replacing whichever relation is first by a corresponding cord. Note that this will change the numbering of the main string, but not the total number of vertices that appear before the start of a relation, as we work our way through the quiver.

The relation which begins in what was originally vertex $n_i$ will be replaced by a cord of length $l_i-2$, and the total number of vertices (counting the main string and cords) before the start of this cord is $n_i$. The preceding cord has length $l_{i-1}-2$ and a total of $n_{i-1}$ vertices before it. So the number of vertices strictly between these two cords is
$$n_i-(n_{i-1}+1)-(l_{i-1}-2)=n_i+1-n_{i-1}-l_{i-1}.$$
Finally, when all relations have been replaced by cords, the last cord has $n_r$ vertices before it and contains $l_r-1$ vertices, including the beginning. Since the total number of vertices in the quiver is $n$, this means that the number of vertices after the last cord is $n+1-n_r-l_r$.

In other words, the quiver we end up with is precisely the quipu quiver 
$$D_{(\hspace{0.35em} n_0 \hspace{0.35em},\ (n_1+1-n_0-l_0),\ \ldots\ ,\ (n_r + 1 - n_{r-1}-l_{r-1}),\ (n + 1 - n_r-l_r))}^{(l_0-2,\hspace{1.75em}  l_1-2 \hspace{1.75em},\ \ldots\ , \hspace{2.6em}  l_r-2 \hspace{2.6em})},$$
and since there is a sequence of mutations between these quivers, their path algebras are derived equivalent. 
\end{proof}

In \cref{fig:quipuLineCorrespondenceProcedure} we have included a concrete example of the approach used in this proof, which step by step transforms the quiver of a Nakayama algebra into a quipu quiver.

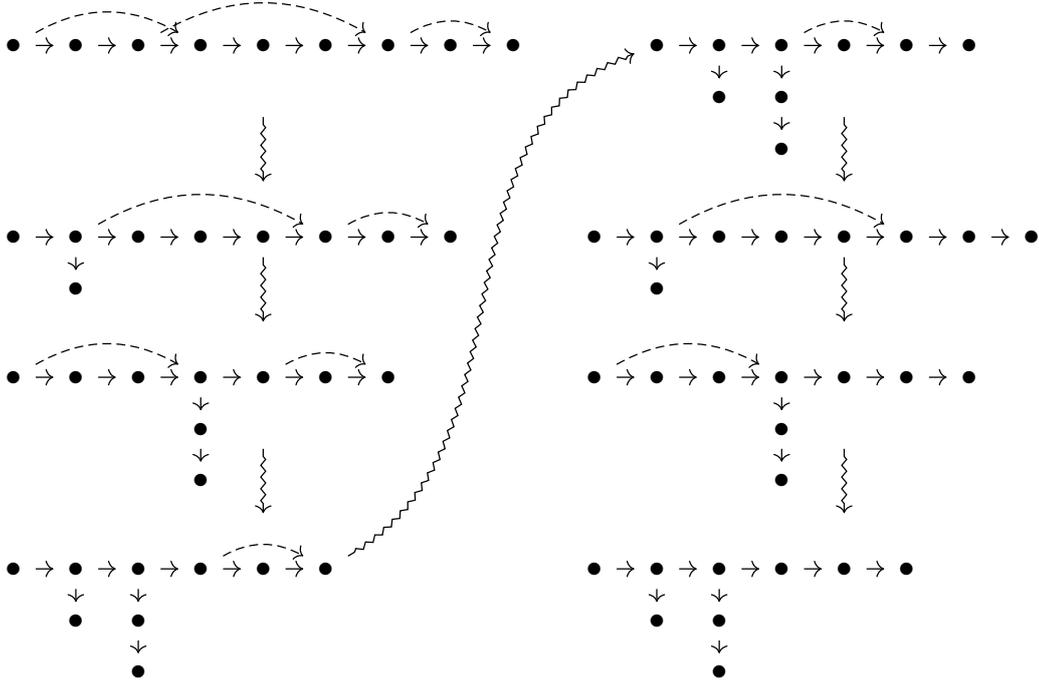
\begin{figure}[H]
\begin{tikzcd}[column sep = tiny, row sep = tiny]
\bullet \arrow[r] \arrow[rrr, dashed, bend left] & \bullet \arrow[r] & \bullet \arrow[r] \arrow[rrrr, dashed, bend left] & \bullet \arrow[r]  & \bullet \arrow[r] & \bullet \arrow[r] & \bullet \arrow[r] \arrow[rr, dashed, bend left] & \bullet \arrow[r] & \bullet & & & \bullet \arrow[r] & \bullet \arrow[r] \arrow[d] & \bullet \arrow[r] \arrow[d] \arrow[rr, dashed, bend left] & \bullet \arrow[r] & \bullet \arrow[r] & \bullet & \\
 &  & & &  {} \arrow[dd, squiggly] & &  &  & & & & & \bullet & \bullet \arrow[d] & {}\arrow[dd, squiggly]  & &  &  \\
 &  & &  &  &  &  &  &  &  & &  & & \bullet & & & &   \\ 
& & & & {} & & & & & &  & & & & {} & & &\\
\bullet \arrow[r] & \bullet \arrow[r] \arrow[d] \arrow[rrrr, dashed, bend left] & \bullet \arrow[r]  & \bullet \arrow[r] & \bullet \arrow[r] \arrow[dd, squiggly] & \bullet \arrow[r] \arrow[rr, dashed, bend left] & \bullet \arrow[r] & \bullet & & & \bullet \arrow[r] & \bullet \arrow[r] \arrow[d] \arrow[rrrr, bend left, dashed] & \bullet \arrow[r] & \bullet \arrow[r] & \bullet \arrow[r] \arrow[dd, squiggly]  & \bullet \arrow[r] & \bullet \arrow[r] & \bullet \\
& \bullet &  & & & &  &  &  && &  \bullet & &  & & & &\\
& & & & {} & & & & & & & & & & {}  & & &\\
 \bullet \arrow[r] \arrow[rrr, dashed, bend left] & \bullet \arrow[r] & \bullet \arrow[r] & \bullet \arrow[r] \arrow[d] & \bullet \arrow[r] \arrow[rr, dashed, bend left] & \bullet \arrow[r] & \bullet & &  & & \bullet \arrow[r] \arrow[rrr, bend left, dashed]& \bullet \arrow[r] & \bullet \arrow[r] & \bullet \arrow[r] \arrow[d] & \bullet \arrow[r] & \bullet \arrow[r] & \bullet & \\
 &  & & \bullet \arrow[d] & {} \arrow[dd, squiggly] & &  & &  & & & & & \bullet \arrow[d] & {} \arrow[dd, squiggly] & & & \\
 &  & & \bullet & &  &  & & & & & & & \bullet &  &  & & \\
 & & & & {} & & & & & & & & & & {} & & & \\
 \bullet \arrow[r] & \bullet \arrow[r] \arrow[d] & \bullet \arrow[r] \arrow[d] & \bullet \arrow[r] \arrow[rr, dashed, bend left] & \bullet \arrow[r] & \bullet \arrow[uuuuuuuuuuurrrrrr, squiggly, out=30, in=200]& & & & & \bullet \arrow[r] & \bullet \arrow[r] \arrow[d] & \bullet \arrow[r] \arrow[d] & \bullet \arrow[r] & \bullet \arrow[r] & \bullet& & \\
 &  \bullet & \bullet \arrow[d] & & & & & &  & & &  \bullet & \bullet \arrow[d] &  & & & &\\
 &  & \bullet & &  &  &  & & & & &  & \bullet &  & &  &  &
\end{tikzcd}
\caption{Going from the quiver of $A_{9,(1,3,7)}^{(3,4,2)}$ to $D_{(1,0,3)}^{(1,2)}$. Each squiggly arrow indicates that we apply the CR-swap to the leftmost relation in the current quipu quiver.}
\label{fig:quipuLineCorrespondenceProcedure}
\end{figure}

\begin{corollary}\label{corollary:2Rels}
If $A_{n,(n_0,n_1, \ldots, n_{r} )}^{(l_{0}, l_{1}, \ldots, l_{r})}$ is a linear Nakayama algebra with almost separate relations, and if $l_i=2$ for some $0\leq i\leq r$, then $A_{n,(n_0,n_1,\ldots, n_i,\ldots, n_{r} )}^{(l_{0}, l_{1}, \ldots, l_i, \ldots, l_{r})}$ is derived equivalent to $A_{n,(n_0,n_1,\ldots, n_{i-1}, n_{i+1},\ldots, n_{r} )}^{(l_{0}, l_{1}, \ldots, l_{i-1},l_{i+1}, \ldots, l_{r})}$.
\end{corollary}

\begin{proof}
Both these algebras satisfy the conditions in \cref{thm:AnToQuipu}, and it's easy to see that the resulting quipus will be identical. This is because the only potential difference is the cord corresponding to the relation starting in $n_i$, but that cord has length $l_i-2=0$, meaning there is no cord there in both cases. So the two Nakayama algebras are derived equivalent to the same quipu algebra, and hence to each other. 
\end{proof}

\Cref{corollary:2Rels} implies that relations of length $2$ have no effect on the derived equivalence class of a linear Nakayama algebra with almost separate relations.

\begin{lemma}\label{lemma:TreeOrientation}
Two quipu quivers without relations have derived equivalent path algebras if and only if they are different orientations of the same underlying quipu.
\end{lemma}

\begin{proof}
From theorem $1.2$ in \cite{AO2013} we know that if $Q$ is a quiver whose underlying graph is a tree, and $Q'$ is a different orientation of the same tree, then $kQ$ and $kQ'$ are derived equivalent. Because quipus are trees, one direction of the lemma follows.

The other direction follows from the proof of corollary 4 in \cite{CK2006}.
\end{proof}

\begin{theorem}\label{thm:QuipuToAn}
Let $P_{(k_0,k_1,\ldots,k_r,k_{r+1})}^{(m_0,m_1,\ldots,m_r)}$ be a quipu, and let 
$$n_i = k_0 + \sum\limits_{j=1}^{i} (m_{j-1} + k_{j} + 1) \text{ for } 1\leq i\leq r+1.$$ 
If $D$ is a quipu quiver, then there exists a derived equivalence
$$kD\simeq A_{n_{r+1},(k_0,n_1,n_2,\ldots, n_{r} )}^{(m_{0}+2, m_{1}+2, \ldots, m_{r}+2)}$$
if and only if $D$ is $P_{(k_0,k_1,\ldots,k_r,k_{r+1})}^{(m_0,m_1,\ldots,m_r)}$ with some orientation.
\end{theorem}

\begin{proof}
We will show that the derived equivalence exists for one orientation of the quipu, namely $D_{(k_0,k_1,\ldots,k_r,k_{r+1})}^{(m_0,m_1,\ldots,m_r)}$. The rest of the theorem follows from \cref{lemma:TreeOrientation}.
If we set 
\begin{align*}
m_i &= l_i - 2 \\
k_0 &= n_0 \\
k_i &= n_i + 1 - n_{i-1} - l_{i-1}, \quad 1\leq i\leq r+1,
\end{align*}
then we know from \cref{thm:AnToQuipu} that there is a derived equivalence
$$kD_{(k_0,k_1,\ldots,k_r,k_{r+1})}^{(m_0,m_1,\ldots,m_r)}\simeq A_{n,(n_0,n_1,\ldots, n_{r} )}^{(l_{0}, l_{1}, \ldots, l_{r})}$$
So we get that $n_0=k_0$ and $l_i=m_i+2$ for all $i$, and solving the last equation for $n_i$ gives $n_i=n_{i-1}+l_{i-1}+k_{i}-1$ for $1\leq i\leq r+1$. This implies that $n_1=k_0+m_{0}+k_1+1$, and an easy induction then shows that 
$$n_i = k_0 + \sum\limits_{j=1}^{i} (m_{j-1} + k_{j} + 1) \text{ for } 1\leq i\leq r+1.$$ 
In particular, we have that
$$n_{r+1} = k_0 + \sum\limits_{j=1}^{r+1} (m_{j-1} + k_{j} + 1) = r+1+\sum\limits_{i=0}^{r+1}k_i + \sum\limits_{i=0}^r m_i,$$
which is equal to the number of vertices in $D_{(k_0,k_1,\ldots,k_r,k_{r+1})}^{(m_0,m_1,\ldots,m_r)}$.
If we now plug these values into $A_{n,(n_0,n_1,n_2,\ldots, n_{r} )}^{(l_{0}, l_{1}, \ldots, l_{r})}$, we get that there is a derived equivalence 
$$kD_{(k_0,k_1,\ldots,k_r,k_{r+1})}^{(m_0,m_1,\ldots,m_r)}\simeq A_{n_{r+1},(k_0,n_1,n_2,\ldots, n_{r} )}^{(m_{0}+2, m_{1}+2, \ldots, m_{r}+2)},$$
which is what we wanted to show. The above Nakayama algebra is derived equivalent to the path algebra of one orientation of the quipu $P_{(k_0,k_1,\ldots,k_r,k_{r+1})}^{(m_0,m_1,\ldots,m_r)}$, and hence, by \cref{lemma:TreeOrientation} it is derived equivalent to the path algebras of all possible orientations of the quipu.
\end{proof}

With \cref{thm:AnToQuipu} and \cref{thm:QuipuToAn} we have explicit formulas to translate between quipu algebras and linear Nakayama algebras with almost separate relations, in a way that preserves derived equivalence.

\begin{corollary}\label{cor:EquivNakayamaAlgebras}
If $A_{n,(n_0,\ldots,n_r)}^{(l_0,\ldots, l_r)}$ is a linear Nakayama algebra with almost separate relations, where $l_i\geq 3$ for all $i$, then the following Nakayama algebras are derived equivalent.
\begin{align*}
&A_{n,(n_0,n_1,\ldots,n_{r-1},n_r)}^{(l_0, l_1, \ldots, l_{r-1}, l_r)}  &&A_{n,(n-n_r-l_r+1,n-n_{r-1}-l_{r-1}+1,\ldots,n-n_0-l_0+1)}^{(l_r, l_{r-1}, \ldots, l_{1}, l_0)}\\
&A_{n,(n_0,n_1,\ldots,n_{r-1},n_r)}^{(l_0, l_1, \ldots, l_{r-1}, n-n_r-l_r+3)} &&A_{n,(l_r-2,n-n_{r-1}-l_{r-1}+1,\ldots,n-n_0-l_0+1)}^{(n-n_r-l_r+3, l_{r-1}, \ldots, l_{1}, l_0)}    \\
&A_{n,(l_0-2,n_1,\ldots,n_{r-1},n_r)}^{(n_0+2, l_1, \ldots, l_{r-1}, l_r)}  && A_{n,(n-n_r-l_r+1,n-n_{r-1}-l_{r-1}+1,\ldots,n-n_0-l_0+1)}^{(l_r, l_{r-1}, \ldots, l_{1}, n_0+l_0)}\\
&A_{n,(l_0-2,n_1,\ldots,n_{r-1},n_r)}^{(n_0+2, l_1, \ldots, l_{r-1}, n-n_r-l_r+3)} &&A_{n,(l_r-2,n-n_{r-1}-l_{r-1}+1,\ldots,n-n_0-l_0+1)}^{(n-n_r-l_r+3, l_{r-1}, \ldots, l_{1}, n_0+l_0)}. 
\end{align*}
Also, adding a set of almost separate relations of length $2$ to available positions in any of these algebras will not change the derived equivalence class.
\end{corollary}

\begin{proof}
We can translate $A_{n,(n_0,\ldots,n_r)}^{(l_0,\ldots, l_r)}$ into a quipu using \cref{thm:AnToQuipu}. Then, the other Nakayama algebras are obtained by applying  \cref{thm:QuipuToAn} to the other orientations of this quipu, given by \cref{obs:QuipuOrientations}. The last part follows from \cref{corollary:2Rels}.
\end{proof}

\begin{remark}
Any quipu quiver with almost separate relations can be viewed as an intermediate step between a Nakayama algebra with almost separate relations and a quipu quiver with no relations. Hence, this derived classification also covers all quipu quivers with almost separate relations, although the combinatorics for describing these become much more involved.
\end{remark}

\section{Complete classification for $n\leq 8$}
We will now explicitly write out the complete classification up to derived equivalence of all linear Nakayama algebras of length $n\leq 8$ with almost separate relations. 
From \cref{corollary:2Rels} we know that we can limit our considerations to algebras $A_{n,(n_0,\ldots,n_r)}^{(l_0,\ldots, l_r)}$ where $l_i\geq 3$ for all $i$. 

We now make a list of every quipu of order $\leq 8$. For each quipu we choose an orientation to get a quipu quiver, and we apply \cref{thm:QuipuToAn} to get the corresponding linear Nakayama algebra. Finally we apply \cref{cor:EquivNakayamaAlgebras} to get the other Nakayama algebras in the same class.

We stopped at $n=8$ for space concerns, but we could easily continue the list for as large $n$ as we need. The following observation is an easy consequence of \cref{cor:EquivNakayamaAlgebras}
\begin{observation}
In each derived equivalence class there are at most $8$ different such algebras whose relations all have length $\geq 3$. 
\end{observation}

These derived equivalence classes will also contain the same algebras decorated with all possible sets of relations of length $2$.

Note that most of these classes will contain other linear Nakayama algebras, with bigger overlap between their relations, but in this paper we focus only on Nakayama algebras with almost separate relations.

In the following table, the quipus that are equal to (extended) Dynkin diagrams are labelled as such, for the rest we use our notation for quipus.

\begin{figure}[H]
\begin{tabular}{|cc|c|}
\hline
\multicolumn{2}{|c|}{\textbf{Quipu}} & \textbf{Corresponding Nakayama algebras} \\
\hline 
\rowcolor{gray!20} $\mathbb{A}_1$ &
    \begin{tikzpicture}[scale=.4]
    \draw[xshift=0 cm,thick,fill=black] (0 cm,0) circle (.3cm);
    \end{tikzpicture}
     & $A_1$ \\
\hline 
$\mathbb{A}_2$ &
    \begin{tikzpicture}[scale=.4]
    \foreach \x in {0,...,1}
    \draw[xshift=\x cm,thick,fill=black] (\x cm,0) circle (.3cm);
    \draw[xshift=0.15 cm,thick] (0.15 cm,0) -- +(1.4 cm,0);
    \end{tikzpicture}
     & $A_2$ \\
\hline 
\rowcolor{gray!20} $\mathbb{A}_3$ &
    \begin{tikzpicture}[scale=.4]
    \foreach \x in {0,...,2}
    \draw[xshift=\x cm,thick,fill=black] (\x cm,0) circle (.3cm);
    \foreach \y in {0.15,...,1.15}
    \draw[xshift=\y cm,thick] (\y cm,0) -- +(1.4 cm,0);
    \end{tikzpicture}
     & $A_3$ \\
\hline 
$\mathbb{A}_4$ &
    \begin{tikzpicture}[scale=.4]
    \foreach \x in {0,...,3}
    \draw[xshift=\x cm,thick,fill=black] (\x cm,0) circle (.3cm);
    \foreach \y in {0.15,...,2.15}
    \draw[xshift=\y cm,thick] (\y cm,0) -- +(1.4 cm,0);
    \end{tikzpicture}
     & $A_4$ \\
\hline
$\mathbb{D}_4$ &
    \begin{tikzpicture}[scale=.4]
    \foreach \x in {0,...,2}
    \draw[thick,xshift=\x cm, fill=black] (\x cm,0) circle (3 mm);
    \foreach \y in {0,...,1}
    \draw[thick,xshift=\y cm] (\y cm,0) ++(.3 cm, 0) -- +(14 mm,0);
    \draw[thick, fill=black] (2 cm,1 cm) circle (3 mm);
    \draw[thick] (2 cm, 3mm) -- +(0, 0.7 cm);
    \end{tikzpicture}
     & $A_{4,(1)}^{(3)}$ \\
\hline 
\rowcolor{gray!20} $\mathbb{A}_5$ &
    \begin{tikzpicture}[scale=.4]
    \foreach \x in {0,...,4}
    \draw[xshift=\x cm,thick,fill=black] (\x cm,0) circle (.3cm);
    \foreach \y in {0.15,...,3.15}
    \draw[xshift=\y cm,thick] (\y cm,0) -- +(1.4 cm,0);
    \end{tikzpicture}
    & $A_5$ \\
\hline 
\rowcolor{gray!20} $\mathbb{D}_5$ &
    \begin{tikzpicture}[scale=.4]
    \foreach \x in {0,...,3}
    \draw[thick,xshift=\x cm, fill=black] (\x cm,0) circle (3 mm);
    \foreach \y in {0,...,2}
    \draw[thick,xshift=\y cm] (\y cm,0) ++(.3 cm, 0) -- +(14 mm,0);
    \draw[thick, fill=black] (2 cm,1 cm) circle (3 mm);
    \draw[thick] (2 cm, 3mm) -- +(0, 0.7 cm);
    \end{tikzpicture}
    & $A_{5,(1)}^{(3)}$, $A_{5,(1)}^{(4)}$, $A_{5,(2)}^{(3)}$\\
\hline 
$\mathbb{A}_6$ &
    \begin{tikzpicture}[scale=.4]
    \foreach \x in {0,...,5}
    \draw[xshift=\x cm,thick,fill=black] (\x cm,0) circle (.3cm);
    \foreach \y in {0.15,...,4.15}
    \draw[xshift=\y cm,thick] (\y cm,0) -- +(1.4 cm,0);
    \end{tikzpicture}
    & $A_6$\\
\hline 
$\mathbb{D}_6$ &
    \begin{tikzpicture}[scale=.4]
    \foreach \x in {0,...,4}
    \draw[thick,xshift=\x cm, fill=black] (\x cm,0) circle (3 mm);
    \foreach \y in {0,...,3}
    \draw[thick,xshift=\y cm] (\y cm,0) ++(.3 cm, 0) -- +(14 mm,0);
    \draw[thick, fill=black] (2 cm,1 cm) circle (3 mm);
    \draw[thick] (2 cm, 3mm) -- +(0, 0.7 cm);
    \end{tikzpicture}
    & $A_{6,(1)}^{(3)}$, $A_{6,(1)}^{(5)}$, $A_{6,(3)}^{(3)}$ \\
\hline  
$\mathbb{E}_6$ &
    \begin{tikzpicture}[scale=.4]
    \foreach \x in {0,...,4}
    \draw[thick,xshift=\x cm, fill=black] (\x cm,0) circle (3 mm);
    \foreach \y in {0,...,3}
    \draw[thick,xshift=\y cm] (\y cm,0) ++(.3 cm, 0) -- +(14 mm,0);
    \draw[thick, fill=black] (4 cm,1 cm) circle (3 mm);
    \draw[thick] (4 cm, 3mm) -- +(0, 0.7 cm);
    \end{tikzpicture}
    & $A_{6,(1)}^{(4)}$, $A_{6,(2)}^{(3)}$, $A_{6,(2)}^{(4)}$ \\
\hline 
$\tilde{\mathbb{D}}_5$ &
    \begin{tikzpicture}[scale=.4]
    \foreach \x in {0,...,3}
    \draw[thick,xshift=\x cm, fill=black] (\x cm,0) circle (3 mm);
    \foreach \y in {0,...,2}
    \draw[thick,xshift=\y cm] (\y cm,0) ++(.3 cm, 0) -- +(14 mm,0);
    \draw[thick, fill=black] (2 cm,1 cm) circle (3 mm);
    \draw[thick] (2 cm, 3mm) -- +(0, 0.7 cm);
    \draw[thick, fill=black] (4 cm,1 cm) circle (3 mm);
    \draw[thick] (4 cm, 3mm) -- +(0, 0.7 cm);
    \end{tikzpicture}
    & $A_{6,(1,3)}^{(3,3)}$ \\
\hline 
\rowcolor{gray!20} $\mathbb{A}_7$ &
    \begin{tikzpicture}[scale=.4]
    \foreach \x in {0,...,6}
    \draw[xshift=\x cm,thick,fill=black] (\x cm,0) circle (.3cm);
    \foreach \y in {0.15,...,5.15}
    \draw[xshift=\y cm,thick] (\y cm,0) -- +(1.4 cm,0);
    \end{tikzpicture}
     & $A_7$ \\
\hline
\rowcolor{gray!20} $\mathbb{D}_7$ &
    \begin{tikzpicture}[scale=.4]
    \foreach \x in {0,...,5}
    \draw[thick,xshift=\x cm, fill=black] (\x cm,0) circle (3 mm);
    \foreach \y in {0,...,4}
    \draw[thick,xshift=\y cm] (\y cm,0) ++(.3 cm, 0) -- +(14 mm,0);
    \draw[thick, fill=black] (2 cm,1 cm) circle (3 mm);
    \draw[thick] (2 cm, 3mm) -- +(0, 0.7 cm);
    \end{tikzpicture}
     & $A_{7,(1)}^{(3)}$, $A_{7,(1)}^{(6)}$, $A_{7,(4)}^{(3)}$\\
\hline 
\rowcolor{gray!20} $\mathbb{E}_7$ &
    \begin{tikzpicture}[scale=.4]
    \foreach \x in {0,...,5}
    \draw[thick,xshift=\x cm, fill=black] (\x cm,0) circle (3 mm);
    \foreach \y in {0,...,4}
    \draw[thick,xshift=\y cm] (\y cm,0) ++(.3 cm, 0) -- +(14 mm,0);
    \draw[thick, fill=black] (4 cm,1 cm) circle (3 mm);
    \draw[thick] (4 cm, 3mm) -- +(0, 0.7 cm);
    \end{tikzpicture}
    & $A_{7,(1)}^{(4)}$, $A_{7,(1)}^{(5)}$, $A_{7,(2)}^{(3)}$, $A_{7,(2)}^{(5)}$, $A_{7,(3)}^{(3)}$, $A_{7,(3)}^{(4)}$\\
\hline 
\rowcolor{gray!20} $\tilde{\mathbb{D}}_6$ &
    \begin{tikzpicture}[scale=.4]
    \foreach \x in {0,...,4}
    \draw[thick,xshift=\x cm, fill=black] (\x cm,0) circle (3 mm);
    \foreach \y in {0,...,3}
    \draw[thick,xshift=\y cm] (\y cm,0) ++(.3 cm, 0) -- +(14 mm,0);
    \draw[thick, fill=black] (2 cm,1 cm) circle (3 mm);
    \draw[thick] (2 cm, 3mm) -- +(0, 0.7 cm);
    \draw[thick, fill=black] (6 cm,1 cm) circle (3 mm);
    \draw[thick] (6 cm, 3mm) -- +(0, 0.7 cm);
    \end{tikzpicture}
    & $A_{7,(1,4)}^{(3,3)}$\\
\hline 
\rowcolor{gray!20} $\tilde{\mathbb{E}}_6$ &
    \begin{tikzpicture}[scale=.4]
    \foreach \x in {0,...,4}
    \draw[thick,xshift=\x cm, fill=black] (\x cm,0) circle (3 mm);
    \foreach \y in {0,...,3}
    \draw[thick,xshift=\y cm] (\y cm,0) ++(.3 cm, 0) -- +(14 mm,0);
    \draw[thick, fill=black] (4 cm,1 cm) circle (3 mm);
    \draw[thick, fill=black] (4 cm,2 cm) circle (3 mm);
    \draw[thick] (4 cm, 3mm) -- +(0, 1.7 cm);
    \end{tikzpicture}
    & $A_{7,(2)}^{(4)}$ \\
\hline
\rowcolor{gray!20} $P_{(1,0,2)}^{(1,1)}$ &
    \begin{tikzpicture}[scale=.4]
    \foreach \x in {0,...,4}
    \draw[thick,xshift=\x cm, fill=black] (\x cm,0) circle (3 mm);
    \foreach \y in {0,...,3}
    \draw[thick,xshift=\y cm] (\y cm,0) ++(.3 cm, 0) -- +(14 mm,0);
    \draw[thick, fill=black] (2 cm,1 cm) circle (3 mm);
    \draw[thick] (2 cm, 3mm) -- +(0, 0.7 cm);
    \draw[thick, fill=black] (4 cm,1 cm) circle (3 mm);
    \draw[thick] (4 cm, 3mm) -- +(0, 0.7 cm);
    \end{tikzpicture}
    & $A_{7,(1,3)}^{(3,3)}$, $A_{7,(1,3)}^{(3,4)}$, $A_{7,(1,4)}^{(4,3)}$, $A_{7,(2,4)}^{(3,3)}$\\
\hline 
$\mathbb{A}_8$ &
    \begin{tikzpicture}[scale=.4]
    \foreach \x in {0,...,7}
    \draw[xshift=\x cm,thick,fill=black] (\x cm,0) circle (.3cm);
    \foreach \y in {0.15,...,6.15}
    \draw[xshift=\y cm,thick] (\y cm,0) -- +(1.4 cm,0);
    \end{tikzpicture}
     & $A_8$ \\
\hline
$\mathbb{D}_8$ &
    \begin{tikzpicture}[scale=.4]
    \foreach \x in {0,...,6}
    \draw[thick,xshift=\x cm, fill=black] (\x cm,0) circle (3 mm);
    \foreach \y in {0,...,5}
    \draw[thick,xshift=\y cm] (\y cm,0) ++(.3 cm, 0) -- +(14 mm,0);
    \draw[thick, fill=black] (2 cm,1 cm) circle (3 mm);
    \draw[thick] (2 cm, 3mm) -- +(0, 0.7 cm);
    \end{tikzpicture}
     & $A_{8,(1)}^{(3)}$, $A_{8,(1)}^{(7)}$, $A_{8,(5)}^{(3)}$\\
\hline 
$\mathbb{E}_8$ &
    \begin{tikzpicture}[scale=.4]
    \foreach \x in {0,...,6}
    \draw[thick,xshift=\x cm, fill=black] (\x cm,0) circle (3 mm);
    \foreach \y in {0,...,5}
    \draw[thick,xshift=\y cm] (\y cm,0) ++(.3 cm, 0) -- +(14 mm,0);
    \draw[thick, fill=black] (4 cm,1 cm) circle (3 mm);
    \draw[thick] (4 cm, 3mm) -- +(0, 0.7 cm);
    \end{tikzpicture}
    & $A_{8,(1)}^{(4)}$, $A_{8,(1)}^{(6)}$, $A_{8,(2)}^{(3)}$, $A_{8,(2)}^{(6)}$, $A_{8,(4)}^{(3)}$, $A_{8,(4)}^{(4)}$\\
\hline 
$\tilde{\mathbb{D}}_7$ &
    \begin{tikzpicture}[scale=.4]
    \foreach \x in {0,...,5}
    \draw[thick,xshift=\x cm, fill=black] (\x cm,0) circle (3 mm);
    \foreach \y in {0,...,4}
    \draw[thick,xshift=\y cm] (\y cm,0) ++(.3 cm, 0) -- +(14 mm,0);
    \draw[thick, fill=black] (2 cm,1 cm) circle (3 mm);
    \draw[thick] (2 cm, 3mm) -- +(0, 0.7 cm);
    \draw[thick, fill=black] (8 cm,1 cm) circle (3 mm);
    \draw[thick] (8 cm, 3mm) -- +(0, 0.7 cm);
    \end{tikzpicture}
    & $A_{8,(1,5)}^{(3,3)}$\\
\hline 
$\tilde{\mathbb{E}}_7$ &
    \begin{tikzpicture}[scale=.4]
    \foreach \x in {0,...,6}
    \draw[thick,xshift=\x cm, fill=black] (\x cm,0) circle (3 mm);
    \foreach \y in {0,...,5}
    \draw[thick,xshift=\y cm] (\y cm,0) ++(.3 cm, 0) -- +(14 mm,0);
    \draw[thick, fill=black] (6 cm,1 cm) circle (3 mm);
    \draw[thick] (6 cm, 3mm) -- +(0, 0.7 cm);
    \end{tikzpicture}
    & $A_{8,(1)}^{(5)}$, $A_{8,(3)}^{(3)}$, $A_{8,(3)}^{(5)}$ \\
\hline
$P_{(1,1,2)}^{(1,1)}$ &
    \begin{tikzpicture}[scale=.4]
    \foreach \x in {0,...,5}
    \draw[thick,xshift=\x cm, fill=black] (\x cm,0) circle (3 mm);
    \foreach \y in {0,...,4}
    \draw[thick,xshift=\y cm] (\y cm,0) ++(.3 cm, 0) -- +(14 mm,0);
    \draw[thick, fill=black] (2 cm,1 cm) circle (3 mm);
    \draw[thick] (2 cm, 3mm) -- +(0, 0.7 cm);
    \draw[thick, fill=black] (6 cm,1 cm) circle (3 mm);
    \draw[thick] (6 cm, 3mm) -- +(0, 0.7 cm);
    \end{tikzpicture}
    & $A_{8,(1,4)}^{(3,3)}$, $A_{8,(1,4)}^{(3,4)}$, $A_{8,(1,5)}^{(4,3)}$, $A_{8,(2,5)}^{(3,3)}$\\
\hline 
$P_{(1,0,3)}^{(1,1)}$ &
    \begin{tikzpicture}[scale=.4]
    \foreach \x in {0,...,5}
    \draw[thick,xshift=\x cm, fill=black] (\x cm,0) circle (3 mm);
    \foreach \y in {0,...,4}
    \draw[thick,xshift=\y cm] (\y cm,0) ++(.3 cm, 0) -- +(14 mm,0);
    \draw[thick, fill=black] (2 cm,1 cm) circle (3 mm);
    \draw[thick] (2 cm, 3mm) -- +(0, 0.7 cm);
    \draw[thick, fill=black] (4 cm,1 cm) circle (3 mm);
    \draw[thick] (4 cm, 3mm) -- +(0, 0.7 cm);
    \end{tikzpicture}
    & $A_{8,(1,3)}^{(3,3)}$, $A_{8,(1,3)}^{(3,5)}$, $A_{8,(1,5)}^{(5,3)}$, $A_{8,(3,5)}^{(3,3)}$\\
\hline 
$P_{(1,0,0,1)}^{(1,1,1)}$ &
    \begin{tikzpicture}[scale=.4]
    \foreach \x in {0,...,4}
    \draw[thick,xshift=\x cm, fill=black] (\x cm,0) circle (3 mm);
    \foreach \y in {0,...,3}
    \draw[thick,xshift=\y cm] (\y cm,0) ++(.3 cm, 0) -- +(14 mm,0);
    \draw[thick, fill=black] (2 cm,1 cm) circle (3 mm);
    \draw[thick] (2 cm, 3mm) -- +(0, 0.7 cm);
    \draw[thick, fill=black] (4 cm,1 cm) circle (3 mm);
    \draw[thick] (4 cm, 3mm) -- +(0, 0.7 cm);
    \draw[thick, fill=black] (6 cm,1 cm) circle (3 mm);
    \draw[thick] (6 cm, 3mm) -- +(0, 0.7 cm);
    \end{tikzpicture}
    & $A_{8,(1,3,5)}^{(3,3,3)}$\\
\hline 
$P_{(2,0,2)}^{(1,1)}$ &
    \begin{tikzpicture}[scale=.4]
    \foreach \x in {0,...,5}
    \draw[thick,xshift=\x cm, fill=black] (\x cm,0) circle (3 mm);
    \foreach \y in {0,...,4}
    \draw[thick,xshift=\y cm] (\y cm,0) ++(.3 cm, 0) -- +(14 mm,0);
    \draw[thick, fill=black] (4 cm,1 cm) circle (3 mm);
    \draw[thick] (4 cm, 3mm) -- +(0, 0.7 cm);
    \draw[thick, fill=black] (6 cm,1 cm) circle (3 mm);
    \draw[thick] (6 cm, 3mm) -- +(0, 0.7 cm);
    \end{tikzpicture}
    & $A_{8,(1,4)}^{(4,3)}$, $A_{8,(1,4)}^{(4,4)}$, $A_{8,(2,4)}^{(3,3)}$, $A_{8,(2,4)}^{(3,4)}$\\
\hline 
$P_{(1,0,2)}^{(1,2)}$ &
    \begin{tikzpicture}[scale=.4]
    \foreach \x in {0,...,4}
    \draw[thick,xshift=\x cm, fill=black] (\x cm,0) circle (3 mm);
    \foreach \y in {0,...,3}
    \draw[thick,xshift=\y cm] (\y cm,0) ++(.3 cm, 0) -- +(14 mm,0);
    \draw[thick, fill=black] (2 cm,1 cm) circle (3 mm);
    \draw[thick] (2 cm, 3mm) -- +(0, 0.7 cm);
    \draw[thick, fill=black] (4 cm,1 cm) circle (3 mm);
    \draw[thick, fill=black] (4 cm,2 cm) circle (3 mm);
    \draw[thick] (4 cm, 3mm) -- +(0, 1.7 cm);
    \end{tikzpicture}
    & $A_{8,(1,3)}^{(3,4)}$, $A_{8,(2,5)}^{(4,3)}$\\
\hline 
$P_{(2,3)}^{(2)}$ &
    \begin{tikzpicture}[scale=.4]
    \foreach \x in {0,...,5}
    \draw[thick,xshift=\x cm, fill=black] (\x cm,0) circle (3 mm);
    \foreach \y in {0,...,4}
    \draw[thick,xshift=\y cm] (\y cm,0) ++(.3 cm, 0) -- +(14 mm,0);
    \draw[thick, fill=black] (4 cm,1 cm) circle (3 mm);
    \draw[thick, fill=black] (4 cm,2 cm) circle (3 mm);
    \draw[thick] (4 cm, 3mm) -- +(0, 1.7 cm);
    \end{tikzpicture}
    & $A_{8,(2)}^{(4)}$, $A_{8,(2)}^{(5)}$, $A_{8,(3)}^{(4)}$\\
\hline 
\end{tabular}
\caption{All quipus of order $\leq 8$, with corresponding derived equivalent Nakayama algebras (not including relations of length $2$).}
\end{figure}

\printbibliography

\end{document}